\documentclass[10pt]{amsart}
\usepackage{latexsym}
\usepackage{euscript}
\usepackage{amsfonts}
\usepackage{amssymb}

\theoremstyle{plain}
\newtheorem{thm}{Theorem}[section]

\newtheorem{lem}[thm]{Lemma}
\newtheorem{cor}[thm]{Corollary}
\newtheorem{prop}[thm]{Proposition}

\theoremstyle{definition}
\newtheorem{Def}[thm]{Definition}
\newtheorem{rem}[thm]{Remark}
\newtheorem{ex}[thm]{Example}

\def\C1{C_{1,\cdot}}

\def\scal#1#2{\langle #1,#2 \rangle}

\title [On the unitary part of isometries ...]{On the unitary part of isometries in commuting, completely non doubly commuting pairs.}
\author{Zbigniew Burdak}
\author{Marek Kosiek}
\author{Marek S\l{}oci\'nski}
\address{Department of Applied Mathematics, University of
Agriculture,\newline ul. Balicka 253c, 30-198 Krakow,
Poland.}
\email{rmburdak@cyf-kr.edu.pl}
\address{Wydzia\l{} Matematyki i Informatyki,
Uniwersytet Jagiello\'nski, ul. Prof. St. \L{}ojasiewicza 6, 30-348 Krak\'ow, Poland}
\email{Marek.Kosiek@im.uj.edu.pl}
\address{Wydzia\l{} Matematyki i Informatyki,
Uniwersytet Jagiello\'nski, ul. Prof. St. \L{}ojasiewicza 6, 30-348 Krak\'ow, Poland}
 \email{Marek.Slocinski@im.uj.edu.pl}

\begin{document}

\maketitle
\begin{abstract}In the paper we describe a unitary extension of any isometry in a commuting, completely non doubly commuting pair of isometries. Precisely we show that Hilbert space of such an extension is a linear span of subspaces reducing  the extension to bilateral shifts.
\end{abstract}
\section{Introduction and preliminaries}
Let $H$ be a complex Hilbert space and $L(H)$ denotes the algebra of all bounded linear operators acting on $H.$  For operator $T\in L(H)$ by negative power $T^n$ we understand $T^{*|n|}.$ For any subspace $L\subset H$ by $P_L$ we understand the orthogonal projection on $L.$ Recall that $L$ reduce $T\in L(H)$ if and only if $T$ commutes with $P_L.$ Recall the classical Wold's result \cite{W}:
\begin{thm}\label{W}
Let $V\in L(H)$ be an isometry. There is a unique decomposition of
$H$ into orthogonal, reducing for $V$ subspaces
\[H=H_u\oplus H_s, \] such that $V|_{H_u}$ is a unitary operator,
$V|_{H_s}$ is a unilateral shift.
 Moreover \[H_u=\bigcap_{n\geq
0}V^nH,\quad H_s=\bigoplus_{n\geq0} V^n\ker V^*. \] $\hfill\Box$
\end{thm}
For a given isometry $V\in L(H)$ subspaces $H_u, H_s$ always mean suitable subspaces in the Wold decomposition given by Theorem \ref{W}. The restrictions $V|_{H_u}, V|_{H_s}$ we call simply a unitary part and a shift part of an isometry.
It arises a natural question about a generalization of Wold result to a pair or a family of operators.  The most natural generalization, which following \cite{CPS} is supposed to be called a multiple canonical Wold decomposition, is obtained in some special cases (\cite{BKS}, \cite{Sl}). In general case of a pair or a family of commuting isometries there are investigated a multiple Wold type decomposition or models (\cite{BDF}, \cite{BDF1}, \cite{BCL}, \cite{B0}, \cite{GG},  \cite{GS}, \cite{Pop}, \cite{Su}.) Recall that operators $T_1, T_2\in L(H)$ doubly commute if they commute and $T_1^*T_2=T_2T_1^*.$ Consider $V_1, V_2\in L(H)$ a pair of isometries. It can be found a maximal subspace reducing it to a doubly commuting pair. In \cite{Sl} is described a multiple Wold decomposition in case of doubly commuting pair. Moreover there is given a model of a pair of doubly commuting unilateral shifts.  Therefore we consider a completely non doubly commuting pairs (i.e. the only subspace of $H$ reducing $V_1, V_2$ to a doubly commuting pair is $\{0\}.$) Examples of such pairs are non doubly commuting unilateral shifts or so called modified bi-shift (see \cite{Pop}).
 Note that if operators commute and one of them is unitary, then they doubly commute. Thus, in considered completely non doubly
 commuting pairs each of the isometry has a nontrivial unilateral shift part and a restriction to any nontrivial subspace reducing
 both operators also have a nontrivial unilateral shift part. However the unitary part may be, but not need to be trivial. There are known examples of undecomposable pairs where unitary part of any isometry is a bilateral shift (mentioned modified bi-shift). In the paper we try to answer the question what kind of unitary operator can be a unitary part of an isometry in a completely non doubly commuting pair.

\section{Multiple Wold decomposition for pairs of isometries}
In the chapter we are going to recall more precisely obtained multiple canonical Wold decomposition and a multiple Wold type decompositions. For a definition of a multiple canonical decomposition in general case refer to \cite{CPS}. In a case of a pair of commuting isometries the definition is as follows:
\begin{Def}
Suppose $V_1,\;V_2$ is a pair of isometries on
$H$. The multiple canonical Wold decomposition is given by a decomposition of a Hilbert space
\[H=H_{uu}\oplus H_{us}\oplus H_{su} \oplus H_{ss},\] where $H_{uu},\; H_{us},\;
H_{su},\; H_{ss}$ are reducing subspaces for $V_1$ and $V_2$ such
that \begin{enumerate}\item[] $V_1|_{H_{uu}},\; V_2|_{H_{uu}}$ are
unitary operators, \item[] $V_1|_{H_{us}}$ is a unitary operator,
$V_2|_{H_{su}}$ is a unilateral shift, \item[] $V_1|_{H_{su}}$ is
a unilateral shift, $V_2|_{H_{su}}$ is a unitary operator,\item[]
$V_1|_{H_{ss}},\; V_2|_{H_{ss}}$ are unilateral
shifts.\end{enumerate}
\end{Def}
 By \cite{Sl} and \cite{BKS}  there are multiple canonical Wold decompositions for doubly commuting pairs of isometries or
 if $\dim \ker V_1^*<\infty$ and $\dim \ker V_2^*<\infty.$  However in the general case we have only a weaker result. Recall a definition from \cite{Pop}.
\begin{Def}
A pair of isometries $V_1,V_2$ is called a weak bi-shift if and
only if $V_1|_{\bigcap_{i\geq 0}\ker V_2^*V_1^i},$
$V_2|_{\bigcap_{i\geq 0}\ker V_1^*V_2^i}$, and the product isometry
$V_1V_2$ are shifts.
\end{Def}
The following general decomposition of a pair of commuting isometries obtained in \cite{Pop} is not a canonical one.
\begin{thm}\label{wbs}  For any pair of commuting
isometries $V_1,V_2$ on $H$ there is a unique decomposition
\[H=H_{uu}\oplus H_{us}\oplus H_{su}\oplus H_{ws},\] such that
$H_{uu},\; H_{us},\; H_{su},\; H_{ws}$ reduce $V_1$ and $V_2$ and
\begin{enumerate}
\item[] $V_1|_{H_{uu}}, V_2|_{H_{uu}}$ are unitary operators,
\item[] $V_1|_{H_{us}}$ is a unitary operator, $V_2|_{H_{us}}$ is
a unilateral shift,\item[] $V_1|_{H_{su}}$ is a unilateral shift,
$V_2|_{H_{su}}$ is a unitary operator, \item[]
$V_1|_{H_{ws}},V_2|_{H_{ws}}$ is a weak bi-shift.\end{enumerate}
\end{thm}$\hfill\Box$

 We are going to focus on a weak bi-shift part. Precisely, we consider a pair of isometries which decomposition given by Theorem
 \ref{wbs} trivializes to a weak bi-shift subspace. In such a case the subspace reducing isometries to a doubly commuting pair is trivial
 or reduces the isometries to a pair of unilateral shifts. Indeed in other case the decomposition of a restriction to a doubly commuting pair of isometries would give a non trivial subspace orthogonal to $H_{ws}.$ By \cite{Pop}  there can be found a maximal subspace of $H_{ws}$ which reduces isometries to a doubly commuting pair of unilateral shifts which model can be found in \cite{Sl}. Therefore we reduce our interest to a completely non doubly commuting pair of isometries. Such pairs are a case of a weak bi-shift which finer but not fully satisfying decomposition has been described in \cite{B0}.
 \section{Decomposition for single isometries}
A unitary part of an isometry in a modified bi-shift is a bilateral shift. Our aim is to describe a unitary part of an isometry in
a pair of commuting, completely non doubly commuting isometries.   In order to answer whether there can be some other operator,
note that being a bilateral shift is not a hereditary property. In other words a restrictions of a bilateral shift to some reducing subspace can be not a bilateral shift.
\begin{ex}\label{e1}
Let $H=L^2(m)$ where $m$ denotes the normalized Lebesgue measure on the unit circle $\mathbb{T}$.
Let $T$ be the operator of multiplication by the variable $''z''$, i.e. $(Tf)(z)=zf(z)$ for $f\in L^2(m)$.
Let $F$ denotes the spectral measure of $T$. Then $F(\alpha)f=\chi_\alpha f$ for $f\in L^2(m)$.
Let $\alpha$ be an entire subarc of $\Gamma$ and $m_\alpha$ be the restriction of $m$ to $\alpha.$
Let $T_\alpha$ be the restriction of $T$ to $F(\alpha)H.$ Since the spectrum $\sigma(T_\alpha)$ do not contain the whole unit circle the restriction $T_\alpha$ is not a bilateral shift.
\end{ex}
Note that a canonical decomposition of an operator $T\in L(H)$ with respect to some property is in fact a construction of a unique
maximal reducing subspace $H_p\subset H$ such that the restriction $T|_{H_p}$ has the property. Since being a bilateral shift is not a hereditary property, there is a problem with construction of a maximal subspace reducing operator to a bilateral shift. Example \ref{e2} in the last section shows that a maximal such subspace is not unique. The construction of a maximal bilateral shift subspace can be done by a construction of a maximal wandering subspace. We follow the idea of wandering vectors from \cite{B0}. Let $G$ be a semigroup and $\{T_g\}_{g\in G}$ be a a semigroup of
isometries on $H$. The vector $x\in H$ is called a \textit{wandering vector} ( for a given semigroup of isometries )
 if for any $g_1,\, g_2\in G$ and $g_1\neq g_2$ holds $(T_{g_1}
x,T_{g_2}x)=0$. For a semigroup generated by by two commuting isometries we obtain the following definition of a wandering vector.

\begin{Def}
A vector $x\in H$ is a wandering vector of isometry $V\in L(H)$ if $V^nx\perp x$ for every positive
$n.$
\end{Def}
Note that for a wandering vector $x$ a vector $x+Vx$ is not wandering. Indeed since $x$ is wandering, then also $Vx$ is wandering. Therefore $\scal{x+Vx}{V(x+Vx)}=\scal{Vx}{Vx}.$ Thus the only linear $V$ invariant subspace of wandering vectors is the trivial one. Since we are interested in reducing subspaces, there is no point in considering subspaces of wandering vectors. Instead we consider
 subspaces generated by wandering vectors.
 Note that for a wandering vector $x$ we have $V^nx\perp V^mx$ for $n\neq m$ but only positive. Let $H=H_u\oplus H_s$ denotes the Wold decomposition  of a given isometry $V\in L(H).$  If a wandering vector $x\in H_u$
then $V^nx\perp V^mx$ for $n,m\in\mathbb{Z}$.  However for $x\in H_s$ is not so clear.   On the other hand every vector in the set
$\bigcup_{n\ge 0}V^n(\ker T^*)$ is wandering, fulfill the orthogonality condition also for negative powers and generates the whole $H_s$. Therefore a ''weaker'' definition of a wandering vector seems to be sufficient.

 \begin{thm}\label{decthm}
For any isometry  $V\in L(H)$ there is a unique decomposition:
$$H=H_0\oplus H_w,$$
reducing operator $V$ such, that
\begin{itemize}
\item $H_w$ is linear span of wandering vectors,
\item $H_0\subset H_u.$
\end{itemize}
\end{thm}
\begin{proof}
Since $H_s\subset H_w$ we need to show only that $H_w$ is $V$ reducing. Obviously $H_w$ is $V$ invariant. Note that for $w$
wandering holds true $w, P_{H_s}w\in H_w$ and consequently also $P_{H_u}w=w-P_{H_s}w\in H_w.$ Note also that
$P_{H_u}H_w=P_{H_u}(H_w\ominus H_s)=H_w\ominus H_s.$  Thus vector $x\in H_w$ if and only if $P_{H_u}x\in H_w.$ Let us
show that $P_{H_u}V^*w\in H_w$ for arbitrary wandering vector $w.$ Denote $w_u=P_{H_u}w, w_s=P_{H_s}w$ for any vector
$w\in H.$ By $\scal{V^nw}{w}=\scal{V^nw_u}{w_u}+\scal{V^nw_s}{w_s}$ vector $w$ is wandering if and only if
$\scal{V^nw_u}{w_u}=-\scal{V^nw_s}{w_s}$ for every positive $n.$ On the other hand
$\scal{V^nV^*w_u}{V^*w_u}=\scal{V^nw_u}{w_u}=-\scal{V^nw_s}{w_s}.$ Thus if $w=w_u\oplus w_s$ is wandering then
$\tilde{w}=(V^*w_u)\oplus w_s$ is wandering as well.  Moreover $P_{H_u}V^*w=V^*P_{H_u}w=P_{H_u}\tilde{w}.$ On the other hand
since $\tilde{w}$ is wandering then by previous argumentation $P_{H_u}\tilde{w}\in H_w.$ Consequently $P_{H_u}V^*w\in H_w.$
Since $w$ was taken arbitrary wandering and $H_w$ is linearly spanned by wandering vectors, we get $V^*P_{H_u}H_w\subset H_w.$
By the showed inclusion and inclusion $H_s\subset H_w$ we get $V^*H_w\subset H_w.$  Consequently $H_w$ reduce $V.$
\end{proof}
Note some property of wandering vectors.
\begin{rem}\label{H0*inv}
Let $V, W\in L(H)$ is a pair of commuting isometries. Let $w$ be a wandering vector for isometry $V$. We have
$$\scal{V^nWx}{V^mWx}=\scal{WV^nx}{WV^mx}=\scal{V^nx}{V^mx}$$ for $n,m\in\mathbb Z_+$ and hence
$Wx$ is also a wandering vector for the isometry $V$.

\end{rem}
As an easy corollary we obtain the following.
\begin{cor}
The subspace $H_w$ in Theorem \ref{decthm} is invariant for every isometry  commuting with $V$ and $H_0$ is invariant for adjoint of isometry commuting with $V.$
\end{cor}
\begin{proof}
Let $W$ be any isometry commuting with $V.$ Since $H_w$ is linearly spanned by a set of all $V$ wandering vectors then by Remark \ref{H0*inv} it is $W$ invariant. Consequently $H_0$ is $W^*$ invariant.
\end{proof}

\section{Decomposition for pairs of isometries}
In this section we take advantage of the showed decomposition and construct a decomposition for pairs of isometries.
\begin{thm}\label{H+}
Let $H_w$ be a linear span of vectors wandering for $V_1\in L(H)$. Consider the decomposition $H=H_0\oplus H_w.$  For any isometry  $V_2\in L(H)$ commuting
with $V_1$ the subspace
 $H_0^+:=\bigvee_{n\ge 0}V_2^n H_0$ reduces $V_1, V_2$ and $V_1$ is unitary on it.
 \end{thm}
 \begin{proof}
 It follows directly from the definition of $H_0^+$ that it is $V_1,V_2$ invariant.
 Since $H_0$ reduce $V_1$ to a unitary operator then
 $$V_1^*V_2^nH_0=V_1^*V_2^nV_1V_1^*H_0=V_1^*V_1V_2^nV_1^*H_0=V_2^nH_0.$$
 Hence $H_0^+$ reduces $V_1$. Note that by formula in Wold decomposition (Theorem \ref{W}) subspace $H_u$ is hyperinvariant. Thus, by $H_0\subset H_u$ follows $H_0^+\subset H_u.$ Consequently $H_0^+$ reduces $V_1$ to a unitary operator.

By Remark \ref{H0*inv}, the subspace $H_0^+$ reduces also $V_2$.
\end{proof}
Note that the subspace $H_0^+$ can be bigger that $H_0.$ Moreover, every wandering vector of a unitary operator generates a subspace reducing it to a bilateral shift. Recall Example \ref{e1}. The whole space $H$ is generated by wandering vectors. On the other hand, in a subspace $F(\alpha)H$ there are no vectors wandering for $T_\alpha.$ Therefore, despite $H\ominus H_0^+\subset H_w$, it is possible that operator $V|_{H\ominus H_0^+}$ can be nontrivially decomposed by Theorem \ref{decthm}. However, in such case we can repeat construction of subspace $H_{01}^+$ which is a subspace constructed by Theorem \ref{H+} for a pair $V_1|_{H\ominus H_0^+}, V_2|_{H\ominus H_0^+}.$ Define a sequence $\mathcal{H}_0=\{0\}$ and $\mathcal{H}_n=\mathcal{H}_{n-1}\ominus H_{0n}^+$ for positive $n$ where $H_{0n}^+$ is a subspace constructed by Theorem \ref{H+} for a pair $V_1|_{H\ominus \mathcal{H}_{n-1}}, V_2|_{H\ominus \mathcal{H}_{n-1}}.$ Note that $H_1:=\bigcap_{n=0}^\infty(H\ominus\mathcal{H}_n)$ is a subspace reducing for $V_1, V_2$ and, if $H$ was separable, subspace $H_0=\{0\}$ in a decomposition of $V_1|_{H_1}, V_2|_{H_1}$ by Theorem \ref{H+}. This way we obtain that whole $H_1$ is linearly spanned by vectors wandering for $V_1.$ We can now repeat a construction on $V_1|_{H_1}, V_2|_{H_1}$ to obtain subspace linearly spanned  by vectors wandering for $V_2.$ However the property of being linearly spanned by wandering vectors is not a hereditary property. Thus we obtain a subspace linearly spanned by vectors wandering for $V_2$ but but it might bo no longer linearly spanned by vectors wandering for $V_1.$ We solve the problem by other way.
Recall form \cite{B1}, that a pair of commuting contractions $V_1, V_2$ is called \textit{strongly completely
non unitary} if there is no proper subspace reducing $V_1,
V_2$ and at least one of them to a unitary operator. Moreover, there is a decomposition theorem (\cite{B1}, Thm. 2.1):
\begin{thm}\label{B1}
Let $T_1, T_2\in L(H)$ be a pair of commuting contractions. There
is a unique decomposition
$$H=H_{uu}\oplus H_{u\neg u}\oplus H_{\neg u u}\oplus
H_{\neg(uu)},$$ where subspaces $H_{uu}, H_{u\neg u}, H_{\neg u
u}, H_{\neg(uu)}$ are maximal such, that:
\begin{enumerate}
\item[] $T_1|_{H_{uu}}, T_2|_{H_{uu}}$ are unitary operators,
\item[]  $T_1|_{H_{u\sim u}}$ is a unitary operator,
$T_2|_{H_{u\neg u}}$ is a completely non unitary operator, \item[]
$T_1|_{H_{\neg u u}}$ is a completely non unitary operator,
$T_2|_{H_{\neg u u}}$ is a unitary operator, \item[]
$T_1|_{H_{\neg(uu)}}, T_2|_{H_{\neg(uu)}}$ is a strongly
completely non unitary pair of contractions.
\end{enumerate}
\end{thm}

The theorem in case of a pair of commuting isometries leads us to the following decomposition.
\begin{thm}
Let $V_1, V_2\in L(H)$ be a pair of commuting isometries. There is a decomposition $$H_{uu}\oplus H_{us}\oplus H_{su}\oplus H_S,$$
where \begin{enumerate}
\item $H_{uu}$ is a maximal subspace reducing $V_1, V_2$ to a pair of unitary operators,
\item $H_{us}$ is a maximal subspace reducing $V_1$ to a unitary operator and $V_2$ to a unilateral shift,
\item $H_{su}$ is a maximal subspace reducing $V_1$ to a unilateral shift and $V_2$ to a unitary operator,
\item $H_S$ reduce $V_1, V_2$ and $H_S$ is linearly spanned by vectors wandering for $V_1$ and is linearly spanned by vectors wandering for $V_2.$
\end{enumerate}
\end{thm}
\begin{proof}
Since a completely non unitary isometry is just a unilateral shift, Theorem \ref{B1} applied for isometries gives a decomposition
$H_{uu}\oplus H_{us}\oplus H_{su}\oplus H_{\neg(uu)}$. We need to show that $H_S=H_{\neg(uu)}$ has suitable properties.
We prove it for the operator $V_1$. Let $H_{\neg(uu)}=H_0\oplus H_1$ be a decomposition of $V_1|_{H_{\neg(uu)}},V_2|_{H_{\neg(uu)}}$ obtained by Theorem \ref{decthm} .
By Theorem \ref{H+} the subspace $H_0$ generates a subspace $H_0^+$ which reduces $V_1$ to a unitary operator and reduces $V_2$.
Since $H_{\neg(uu)}$ reduces $V_1, V_2$ it holds true $H_0^+\subset H_{\neg(uu)}$. On the other hand,
$V_1|_{H_{\neg(uu)}}, V_2|_{H_{\neg(uu)}}$ is a strongly completely non unitary pair.
Thus $H_0^+=\{0\}$ and consequently $H_0=\{0\}$. Hence $H_{\neg(uu)}=H_w$ and is linearly spanned by vectors wandering for $V_1.$
\end{proof}

The immediate consequence of the theorem is the following:
\begin{prop}
Let $V_1, V_2\in L(H)$ be a pair of commuting, completely non doubly commuting isometries. There are sets $W_1, W_2$ of vectors wandering for $V_1, V_2$ suitably such that:
$$H=\bigvee\left\{w\in W_i\right\}\quad\text{ for } i=1,2.$$
\end{prop}
 By the proof of Theorem \ref{decthm} the projection of a wandering vector on unitary subspace may not be wandering.
\begin{rem}\label{r2}
Let $V\in L(H)$ be an isometry and $w\in H$ be a wandering vector. Note that for every wandering vector in $H_u$ equality $\scal{V^nP_{H_u}w}{V^mP_{H_u}w}=0$ holds true for every $n,m\in\mathbb{Z}, n\ne m.$ Thus the minimal $V$ reducing subspace generated by $w$ is $\bigoplus_{n\in\mathbb{Z}}V^n\left(\mathbb{C}w\right).$
In other words the minimal $V$ reducing subspace generated by a wandering vector in $H_u$ reduce $V$ to a bilateral shift.
\end{rem}
Let us introduce a definition of some class of operators.
\begin{Def}
We call  an operator $V\in L(H)$ a span of bilateral shifts if there are subspaces $\{H_\iota\}_{\iota\in I}$ such that $H=\bigvee_{\iota\in I}H_\iota$ and $V|_{H_\iota}$ is a bilateral shift.
 \end{Def}
 Note that a span of bilateral shifts is a unitary operator.
\begin{cor}\label{bispan}
Let $V\in L(H)$ be an isometry such that $H$ is a linear span of $V$ wandering vectors. The minimal unitary extension $U\in L(K)$ of $V$ is a span of bilateral shifts.
\end{cor}
\begin{proof}
Note that every  $V$ wandering vector $w\in H$ becomes $U$ wandering. On the other hand $U$ is unitary. According to Remark \ref{r2} subspace $L_w:=\bigoplus_{n\in\mathbb{Z}}V^n\left(\mathbb{C}w\right)$ is the minimal $U$ reducing subspace generated by $w.$
 Since $H$ is linearly spanned by wandering vectors then $H\subset\bigvee_{w\in W}L_w$ where $W$ denotes set of all $V$ wandering vectors. Since $L_w\subset K,$ by minimality of unitary extension we have $K=\bigvee_{w\in W}L_w.$
 On the other hand $U|_{L_w}$ is a bilateral shift which finishes the proof.
 \end{proof}
 Eventually we obtain the following result.
\begin{thm}
Let $V_1, V_2\in L(H)$ be a pair of commuting, completely non doubly commuting isometries. The unitary extension of each
isometry is a span of bilateral shifts.
\end{thm}
\section{Examples}
We are going to give a few examples. The first is an example of isometry $V\in L(H)$ such that $H_s\ne \{0\}$ and $H$ is not linearly spanned by wandering vectors.
\begin{ex}\label{e3}
Denote by $f, e_i$ for $i\in\mathbb{Z}_+$ the set of orthonormal vectors in some Hilbert space. Define new Hilbert space $H:=\mathbb{C}f\oplus\bigoplus_{i\in\mathbb{Z}_+}\mathbb{C}e_i$ and isometry $V\in L(H)$ by $Vf=f, Ve_i=e_{i+1}$ for $i\in\mathbb{Z}_+.$ Assume that $H$ is linearly spanned by vectors wandering for $V.$ Then there is $w$ a wandering vector, such that $P_{H_u}w\ne 0.$ Obviously $H_u=\mathbb{C}f.$ Assume for convenience that $w=f+v$ where $v\in H_s=\bigoplus_{i\in\mathbb{Z}_+}\mathbb{C}e_i.$ By the proof of Theorem \ref{decthm} since $w$ is wandering then $\scal{V^nf}{f}=-\scal{V^nv}{v}.$ By the definition of $V$ we obtain $\|f\|=-\scal{V^nv}{v}=-\scal{v}{V^{*n}v}.$ Since $v\in H_s$ then $V^{*n}v$ vanishes to $0.$ Consequently we obtain a contradiction $1=\|f\|=\lim_{n\to\infty}-\scal{v}{V^{*n}v}=0.$ Thus H can not be linearly spanned by $V$ wandering vectors.\end{ex}

The next is the example of a span of bilateral shifts which is not a bilateral shift. We would like to thanks Professor  L\'aszl\'o K\'erchy for this example.

\begin{ex}\label{e2}
Let $H=L^2(\alpha)\oplus L^2(2\alpha)\oplus L^2(\alpha)$ and $U\in L(H)$ be multiplying by $''z''.$ If we assume $\alpha$ to be such an subarc that $\alpha\cup 2\alpha=\mathbb{T}$ where $\mathbb{T}$ denotes the whole unit circle then $H_0=\{0\}.$
If the operator would be unitary equivalent to some bilateral shift then its spectral multiplicities has to be equal. However spectral multiplicity of a bilateral shift is constant, while in our example it is not.
 \end{ex}
It is clear that wandering vectors of isometry $V\in L(H)$ span the whole subspace $H_s.$ On the other hand by Remark \ref{r2} wandering vector in a subspace $H_u$ fulfills the orthogonality $V^nw\perp V^mw$ for every integer powers. The natural question is what will change if we make a definition of a wandering vector stronger in the following way. We call a wandering vector $w$ strongly wandering if it fulfills the condition $V^nw\perp V^mw$ for every $n,m\in\mathbb{Z}, n\ne m.$ Denote $H_{ws}$ the minimal subspace linearly spanned by strongly wandering vectors and $H_w$ a subspace linearly spanned by wandering vectors. Obviously $H_{ws}\subset H_w$ and both subspaces are reducing for isometry $V.$ As we know subspace $H_w$ is invariant for every isometry commuting with $V$ but $H_{ws}$ does not need to be. Consider the following lemma.
\begin{lem}\label{wand}
Let $V\in L(H)$ be isometry $H=H_u\oplus H_s$ denotes its Wold decomposition. Then the following conditions holds.
\begin{enumerate}
\item for $x$ a strongly wandering vector also $x_u:=P_{H_u}x, x_s:=P_{H_s}x$ are strongly wandering vectors,
\item $W=H_s\oplus W_u$ where $W, W_u$ denotes subspaces generated by strongly wandering vectors for $V$ and $V|_{H_u}$ respectively.
\end{enumerate}
\end{lem}
\begin{proof}
Since $H_u, H_s$ reduce $V$ and $x$ is wandering then $\scal{V^nx_u}{V^mx_u}=-\scal{V^nx_s}{V^mx_s}$ for every $n\ne m.$ On the other hand $$\scal{V^nx_u}{V^mx_u}=\scal{V^{*k}V^nx_u}{V^{*k}V^mx_u}=\lim\limits_{k\to\infty}\scal{V^{n-k}x_u}{V^{m-k}x_u}=\dots$$
for $n\ne m$ also $n-k\ne m-k$ and consequently
$$\dots=\lim\limits_{k\to\infty}-\scal{V^{n-k}x_s}{V^{m-k}x_s}=0.$$ Thus $x_u$ is a wandering vector and by $\scal{V^nx_u}{V^mx_u}=-\scal{V^nx_s}{V^mx_s}$ also $x_s$ is wandering.

For the second part note that $V^n(\ker V^*)$ for every $n\ge 0$ is a set of $V$ wandering vectors. Thus $H_s\subset W.$ Since $H_u$ reduce $V$ every vector wandering for $V|_{H_u}$ is wandering for $V.$ Thus $W_u\subset W.$ By the first part of the lemma follows the reverse inclusion $W\subset H_s\oplus W_u.$
\end{proof}
We want to show that unitary extension can be a linear span of bilateral shifts and Hilbert space $H$ is not linearly spanned by strongly wandering vectors. Consider the following example.
\begin{ex}
 Consider Example \ref{e2} and denote $K=L^2(\alpha)\oplus L^2(2\alpha)\oplus L^2(\alpha)$ and operator of multiplying by $''z''$ by $U.$ Find a wandering subspace $W$ in a bilateral shift $ L(2\alpha)\oplus L(\alpha)$ such that $ L(2\alpha)\oplus L(\alpha)=\bigoplus_{n\in\mathbb{Z}}U^nW.$ Then take $H=L^2(\alpha)\oplus\bigoplus_{n\ge 0}U^n(W).$ Restriction $U|_H$ is isometry with unitary part equal $L^2(\alpha)$ and $U$ is its minimal unitary extension. Since $\sigma(T|_{H_u})$ do not contain unit circle, there is no subspace reducing it to a bilateral shift. Consequently $H_u$ do not contain any wandering vector and $H_{ws}=H_s=\bigoplus_{n\in\mathbb{Z}}U^nW.$ On the other hand the unitary extension is a span of bilateral shifts.
\end{ex}

\end{document}